\documentclass[11pt,reqno]{amsart}

\usepackage[T1]{fontenc}
\usepackage{lmodern}
\usepackage{microtype}
\usepackage{mathtools}
\usepackage{amssymb}
\usepackage[margin=1.12in]{geometry}
\usepackage[colorlinks=true,linkcolor=blue,citecolor=blue,urlcolor=blue]{hyperref}
\usepackage[capitalise,noabbrev]{cleveref}

\newtheorem{theorem}{Theorem}[section]
\newtheorem{corollary}[theorem]{Corollary}
\newtheorem{lemma}[theorem]{Lemma}
\theoremstyle{remark}

\DeclareMathOperator{\distop}{dist}
\DeclareMathOperator{\spanop}{span}
\newcommand{\R}{\mathbb R}
\newcommand{\cH}{\mathcal H}

\title[Property $(T_X)$ for $L_p$-subspaces]
{Kazhdan's Property $(T)$ for Subspaces and Quotients of
$L_p$-Spaces}

\author{Qingjin Cheng}
\address{School of Mathematical Sciences, Xiamen University,
Xiamen 361005, China}
\email{qjcheng@xmu.edu.cn}
\thanks{Qingjin Cheng was supported by NSFC
(Grant No.~12071389).}
\date{}
\hypersetup{
  pdftitle={Kazhdan's Property (T) for Subspaces and Quotients of Lp-Spaces},
  pdfauthor={Qingjin Cheng},
  pdfsubject={Kazhdan property for subspaces and quotients of Lp-spaces},
  pdfkeywords={Kazhdan property (T), Banach property (TX), Lp-space,
  isometric representation, homogeneous polynomial, invariant weight}
}


\begin{document}
\begin{abstract}
We prove a uniform spectral-gap estimate at every even exponent for
isometric representations on arbitrary closed subspaces of real
$L_p$-spaces.  As an application, this result, combined with the work
of Bader, Furman, Gelander and Monod, gives an affirmative answer to
their Question~4.1: Kazhdan's property~$(T)$ implies property~$(T_X)$
whenever $X$ is a closed subspace or a quotient of a real $L_p$-space,
for every $1<p<\infty$.
\end{abstract}

\maketitle
\enlargethispage{5pt}

\section{Introduction and main results}

Kazhdan's property~$(T)$ asserts that unitary representations have a
spectral gap away from their invariant vectors.  We refer to
\cite{BHV} for a general account and to the recent surveys
\cite{BaderICM,deLaSalleICM} for developments involving higher forms
of property~$(T)$, Banach-space representations, and rigidity of
lattices.  Bader, Furman, Gelander and Monod introduced a Banach-space
version of this spectral-gap property in \cite{BFGM}.  For a Banach
space $X$, property~$(T_X)$ requires that, for every strongly
continuous isometric representation $\rho$ on $X$, the induced
representation on $X/X^{\rho(G)}$ does not almost have invariant
vectors; the precise definition is recalled in \Cref{sec:preliminaries}.

Throughout the paper, all Banach spaces are real.  If
$(\Omega,\Sigma,\mu)$ is a measure space, we abbreviate
$L_p(\Omega,\Sigma,\mu)$ to $L_p(\mu)$; an $L_p$-space means a space
of this form.

Let $G$ be a locally compact second countable group.  Bader, Furman,
Gelander and Monod proved that property~$(T)$ implies property~$(T_X)$
when $X$ is a closed subspace of $L_p(\mu)$ and
$p\notin\{4,6,8,\ldots\}$, and when $X$ is a quotient of $L_p(\mu)$ and
$p\notin\{4/3,6/5,8/7,\ldots\}$ \cite[Theorem~A]{BFGM}.  In their
proof, the subspace and quotient statements are deduced from the
$L_p(\mu)$ case using Hardin's extension theorem for isometries defined
on $L_p$-subspaces \cite{Hardin}, in the form recorded in
\cite[Theorem~2.18 and Corollary~2.20]{BFGM}.  As they explicitly note,
the restrictions on the exponents in these statements come from this
use of Hardin's theorem \cite[p.~81]{BFGM}.  They then asked whether
property~$(T)$ implies property~$(T_X)$ for every closed subspace and
every quotient $X$ of $L_p(\mu)$, for every $1<p<\infty$
\cite[Question~4.1]{BFGM}.  The exceptional exponents from their theorem
are still recorded in recent work on representations on
$L_p$-subspaces \cite[p.~6]{BoucherSpakula}.

The following theorem resolves the remaining closed-subspace case in
quantitative form.  At every even exponent, it gives a spectral-gap
estimate that is uniform over all $\sigma$-finite measure spaces,
closed subspaces, and isometric representations.

\begin{theorem}
\label{thm:uniform-subspace}
Let $G$ be a locally compact second countable group with Kazhdan's
property $(T)$, and fix a Hilbert Kazhdan pair $(Q,\kappa_2)$ for $G$.
For every even integer $p\ge2$, there exists a constant
$c_p>0$, depending only on $p$ and $\kappa_2$, such that the following
holds: for every $\sigma$-finite measure space
$(\Omega,\Sigma,\mu)$, every closed subspace
$X\subseteq L_p(\mu)$, every strongly continuous
representation $\rho:G\to O(X)$, and every $x\in X$,
\begin{equation}\label{eq:main-gap}
 \max_{q\in Q}\|\rho(q)x-x\|_p
 \ge
 c_p\,\distop\bigl(x,X^{\rho(G)}\bigr).
\end{equation}
\end{theorem}

Combining the theorem with the nonexceptional case of Bader, Furman,
Gelander and Monod and their duality principle gives the complete
answer to their question.

\begin{corollary}\label{cor:main}
Let $G$ be a locally compact second countable group with Kazhdan's
property $(T)$, let $1<p<\infty$, and let $\mu$ be a $\sigma$-finite
measure on a standard Borel space $\Omega$.  If $X$ is either a closed
subspace or a quotient of $L_p(\mu)$, then
$G$ has property $(T_X)$.
\end{corollary}

Let us indicate the main idea.  At an even exponent $p=2m$, the
ambient extension of an $L_p$-subspace isometry used in \cite{BFGM} is
no longer available.  We instead use the power map $x\mapsto x^m$ to
associate a canonical Hilbert space to the subspace and lift each of
its isometries directly to that Hilbert space.  If $m$ is odd, the
Hilbert spectral gap transfers back immediately.  If $m$ is even, the
power map loses signs.  Projecting onto the invariant vectors produces
a nonnegative invariant weight; this gives a representation on a
closed subspace of a weighted $L_m$-space.  A pointwise lifting
inequality then recovers the original $L_{2m}$-displacement.  Iterating
this descent along the factorization $p=2^a r$, with $r$ odd, proves the
theorem.

\Cref{sec:preliminaries} recalls the necessary definitions and
establishes the estimates used later.  The proofs of
\Cref{thm:uniform-subspace,cor:main} are given in
\Cref{sec:main-proofs}.

\section{Preliminaries}\label{sec:preliminaries}

This section contains the definitions and preliminary results used in
the proofs of \Cref{thm:uniform-subspace,cor:main}.  After recalling
isometric representations and property~$(T_X)$, we discuss
representations modulo fixed vectors, the power-map construction, and
the spectral-gap estimates used to reduce the exponent.

\subsection{Isometric representations and property
\texorpdfstring{$(T_X)$}{(TX)}}

Let $G$ be a locally compact second countable group and let $X$ be a
Banach space.  We write $O(X)$ for the group of all surjective linear
isometries of $X$, equipped with the strong operator topology.
A linear isometric representation of $G$ on $X$ is a group
homomorphism
\[
 \rho:G\longrightarrow O(X).
\]
Such a representation is \emph{strongly continuous} if $\rho$ is
continuous for this topology; equivalently, for every $x\in X$, the
orbit map
\[
 G\longrightarrow X,\qquad g\longmapsto\rho(g)x,
\]
is norm-continuous.
We put
\[
 S(X)=\{x\in X:\|x\|=1\},\qquad
 X^{\rho(G)}=\{x\in X:\rho(g)x=x\text{ for every }g\in G\},
\]
and, for a closed subspace $M\subseteq X$,
\[
 \distop(x,M)=\inf_{a\in M}\|x-a\|.
\]
A closed subspace $M\subseteq X$ is called \emph{$\rho(G)$-invariant} if
$\rho(g)M=M$ for every $g\in G$.  For closed subspaces
$Y,Z\subseteq X$, the notation
\[
 X=Y\oplus Z
\]
means that $X$ is their topological direct sum: every $x\in X$ has a
unique decomposition $x=y+z$ with $y\in Y$ and $z\in Z$, and the
associated coordinate projections are bounded.  For a subset
$A\subseteq X$, we write
\[
 \operatorname{diam}(A)=\sup\{\|x-y\|:x,y\in A\},
\]
with the convention $\operatorname{diam}(\varnothing)=0$.
We also use the extended-real convention $\inf\varnothing=+\infty$.

Let
\[
 q_\rho:X\longrightarrow X/X^{\rho(G)}
\]
be the quotient map.  The representation $\rho$ induces a strongly
continuous linear isometric representation $\bar\rho$ on
$X/X^{\rho(G)}$, characterized by
\[
 \bar\rho(g)q_\rho(x)=q_\rho\bigl(\rho(g)x\bigr)
 \qquad(g\in G,\ x\in X).
\]
Following \cite[(1.i)]{BFGM}, a strongly continuous representation
$\rho:G\to O(X)$ \emph{almost has invariant vectors} if
\begin{equation}\label{eq:almost-invariant}
 \inf_{x\in S(X)}\operatorname{diam}\bigl(\rho(K)x\bigr)=0
 \qquad\text{for every compact set }K\subseteq G.
\end{equation}
Following \cite[Definition~1.1]{BFGM}, $G$ has property
$(T_X)$ if, for every strongly continuous representation
$\rho:G\to O(X)$, the representation $\bar\rho$ does not almost have
invariant vectors.  For Hilbert $X$,
this is precisely Kazhdan's property~$(T)$
\cite[the discussion following Definition~1.1]{BFGM}.

For the quantitative argument, let $Q\subseteq G$ be compact with
$e\in Q$, and let $\kappa_2>0$.  The pair $(Q,\kappa_2)$ is called a
\emph{Hilbert Kazhdan pair} if every strongly continuous orthogonal
representation $\pi:G\to O(H)$ satisfies
\begin{equation}\label{eq:hilbert-gap}
 \max_{q\in Q}\|\pi(q)\xi-\xi\|_H
 \ge \kappa_2\,\distop\bigl(\xi,H^{\pi(G)}\bigr)
 \qquad(\xi\in H).
\end{equation}
The existence of such a pair is equivalent to property~$(T)$
\cite[Definition~1.1.3]{BHV}.

We also recall two geometric notions used below.  The space $X$ is
\emph{uniformly convex} if
\[
 \delta_X(\varepsilon)
 :=\inf\left\{1-\left\|\frac{x+y}{2}\right\|:
 x,y\in S(X),\ \|x-y\|\ge\varepsilon\right\}>0
 \quad(0<\varepsilon\le2),
\]
and it is \emph{uniformly smooth} if
\[
 \lim_{t\downarrow0}\frac{\varrho_X(t)}{t}=0,
 \qquad
 \varrho_X(t):=\sup_{x,y\in S(X)}
 \left(\frac{\|x+ty\|+\|x-ty\|}{2}-1\right).
\]

\subsection{Representations modulo fixed vectors and normalization}

We first introduce the notation for the canonical complement.  Associated
with a representation $\rho:G\to O(X)$ is its contragredient action
\[
 \rho^*:G\longrightarrow O(X^*),\qquad
 \rho^*(g)=\bigl(\rho(g^{-1})\bigr)^*;
\]
equivalently,
\[
 \bigl(\rho^*(g)f\bigr)(x)
 =f\bigl(\rho(g^{-1})x\bigr)
 \qquad(g\in G,\ f\in X^*,\ x\in X).
\]
Its fixed-vector space is
\[
 (X^*)^{\rho^*(G)}
 =\{f\in X^*: \rho^*(g)f=f\text{ for every }g\in G\}.
\]
For a subset $E\subseteq X^*$, write
\[
 {}^\perp E
 :=\{x\in X:f(x)=0\text{ for every }f\in E\}
\]
for its pre-annihilator in $X$, and define
\begin{equation}\label{eq:canonical-complement}
 X'_\rho
 :={}^\perp\bigl((X^*)^{\rho^*(G)}\bigr)
 =\{x\in X:f(x)=0\text{ for every }
       f\in (X^*)^{\rho^*(G)}\}.
\end{equation}

The quotient representation can be realized on the canonical
complement as follows.

\begin{lemma}\label{lem:canonical-complement}
Let $X$ be a uniformly convex Banach space, and let
$\rho:G\to O(X)$ be a strongly continuous linear isometric
representation.  Then $X'_\rho$ is a closed $\rho(G)$-invariant
subspace satisfying
\begin{equation}\label{eq:canonical-splitting}
 X=X^{\rho(G)}\oplus X'_\rho,
 \qquad
 (X'_\rho)^{\rho(G)}=\{0\}.
\end{equation}
Moreover,
\[
 q_\rho|_{X'_\rho}:
 X'_\rho\longrightarrow X/X^{\rho(G)}
\]
is an isomorphism of Banach $G$-representations, and
\[
 \bar\rho\text{ almost has invariant vectors}
 \quad\Longleftrightarrow\quad
 \rho|_{X'_\rho}\text{ almost has invariant vectors}.
\]
\end{lemma}

\begin{proof}
Since $X$ is uniformly convex, it is superreflexive, while the
isometric representation $\rho$ is uniformly equicontinuous.  By
\cite[Proposition~2.6]{BFGM}, $X'_\rho$ is a closed
$\rho(G)$-invariant subspace and
\[
 X=X^{\rho(G)}\oplus X'_\rho.
\]
It follows that
$(X'_\rho)^{\rho(G)}=X'_\rho\cap X^{\rho(G)}=\{0\}$.

The direct-sum decomposition also shows that
\[
 J:=q_\rho|_{X'_\rho}:
 X'_\rho\longrightarrow X/X^{\rho(G)}
\]
is bijective.  The map $J$ is bounded, and its inverse is bounded by
the bounded inverse theorem.  Furthermore,
\[
 J\rho(g)x
 =q_\rho\bigl(\rho(g)x\bigr)
 =\bar\rho(g)q_\rho(x)
 =\bar\rho(g)Jx
 \qquad(g\in G,\ x\in X'_\rho),
\]
so $J$ is an isomorphism of Banach $G$-representations.

If $X'_\rho=\{0\}$, then $X/X^{\rho(G)}=\{0\}$, and the final
assertion follows from the convention $\inf\varnothing=+\infty$.
Suppose that $X'_\rho\ne\{0\}$, and write
$\rho'=\rho|_{X'_\rho}$.  For an isometric representation $\sigma$
on a nonzero Banach space $Y$ and a compact set $K\subseteq G$, put
\[
 \alpha_\sigma(K)
 :=\inf_{y\in S(Y)}\operatorname{diam}\bigl(\sigma(K)y\bigr).
\]
Normalizing the image under any bounded equivariant isomorphism
$T:Y\to Z$ gives
\[
 \alpha_\tau(K)
 \le \|T\|\,\|T^{-1}\|\,\alpha_\sigma(K),
 \qquad
 \alpha_\sigma(K)
 \le \|T\|\,\|T^{-1}\|\,\alpha_\tau(K),
\]
where $\tau$ is the representation on $Z$.  Applying these inequalities
to $T=J$, $\sigma=\rho'$, and $\tau=\bar\rho$ shows that
$\alpha_{\rho'}(K)=0$ if and only if $\alpha_{\bar\rho}(K)=0$ for
every compact $K\subseteq G$.  The conclusion follows from
\Cref{eq:almost-invariant}.
\end{proof}

We use the following normalization repeatedly.

\begin{lemma}
\label{lem:normalization}
Let $X$ be a uniformly convex Banach space, let
$\rho:G\to O(X)$ be a group homomorphism, and let
$x\notin X^{\rho(G)}$.  Then there exists $a\in X^{\rho(G)}$ such that
\[
 \|x-a\|_X
 =\distop\bigl(x,X^{\rho(G)}\bigr)>0
\]
and the vector
\[
 v=\frac{x-a}{\distop\bigl(x,X^{\rho(G)}\bigr)}
\]
satisfies
\begin{equation}\label{eq:normalization}
 \|v\|_X=\distop\bigl(v,X^{\rho(G)}\bigr)=1,
 \qquad
 \sup_{g\in G}\|\rho(g)v-v\|_X\ge1.
\end{equation}
\end{lemma}

\begin{proof}
The fixed-vector space $X^{\rho(G)}$ is closed.  Since $X$ is
reflexive, the distance from $x$ to $X^{\rho(G)}$ is attained at some
$a\in X^{\rho(G)}$; it is positive because
$x\notin X^{\rho(G)}$.  Since $X^{\rho(G)}$ is a linear subspace,
$\|v\|_X=1$ and
\[
 \distop\bigl(v,X^{\rho(G)}\bigr)
 =
 \frac{\distop\bigl(x-a,X^{\rho(G)}\bigr)}
      {\distop\bigl(x,X^{\rho(G)}\bigr)}
 =1.
\]

Let $C=\overline{\operatorname{conv}}\{\rho(g)v:g\in G\}$, where the
closure is taken in norm.  The set $C$ is weakly compact, and the norm
has a unique minimizer $y\in C$, by reflexivity and strict convexity.
For every $h\in G$, the equality $\rho(h)C=C$ and uniqueness of $y$
give $\rho(h)y=y$.  Hence $y\in X^{\rho(G)}$.

Set $r=\sup_{g\in G}\|v-\rho(g)v\|_X$.  The closed ball with center
$v$ and radius $r$ contains the orbit of $v$ and therefore contains
$C$.  It follows that
\[
 1=\distop\bigl(v,X^{\rho(G)}\bigr)
 \le\|v-y\|_X
 \le r,
\]
which proves \Cref{eq:normalization}.
\end{proof}

\subsection{Power maps and Hilbert-space lifts}

Throughout this subsection, $(\Omega,\Sigma,\mu)$ is a measure space,
and $L_r(\mu)$ denotes $L_r(\Omega,\Sigma,\mu)$.  For an integer
$m\ge1$ and $x\in L_{2m}(\mu)$, we write $x^m$ for the pointwise
power $\omega\mapsto x(\omega)^m$.  Then $x^m\in L_2(\mu)$ and
$\|x^m\|_2=\|x\|_{2m}^m$.

We begin with three elementary estimates for the power map.

\begin{lemma}\label{lem:power-estimates}
Let $m\ge1$ be an integer.  For all $x,y\in L_{2m}(\mu)$,
\begin{equation}\label{eq:power-upper-general}
 \|x^m-y^m\|_2
 \le
 m\max\{\|x\|_{2m},\|y\|_{2m}\}^{m-1}
 \|x-y\|_{2m}.
\end{equation}
If $m$ is odd, then
\begin{equation}\label{eq:power-lower}
 \|x^m-y^m\|_2
 \ge 2^{-(m-1)}\|x-y\|_{2m}^m.
\end{equation}
If $m$ is even, then every $a,b\in\R$ satisfy
\begin{equation}\label{eq:scalar-lifting}
 |a-b|^{2m}
 \le
 4^m\bigl(
 |a^m-b^m|^2+|a|^m|a-b|^m
 \bigr).
\end{equation}
\end{lemma}

\begin{proof}
Put $M=\max\{\|x\|_{2m},\|y\|_{2m}\}$.  The factorization
$x^m-y^m=(x-y)\sum_{j=0}^{m-1}x^{m-1-j}y^j$ and H\"older's
inequality give
\[
 \|x^m-y^m\|_2
 \le
 \sum_{j=0}^{m-1}
 \|x\|_{2m}^{m-1-j}\|y\|_{2m}^j\|x-y\|_{2m}
 \le mM^{m-1}\|x-y\|_{2m},
\]
which proves \Cref{eq:power-upper-general}.

Suppose that $m$ is odd.  If $a,b\in\R$ have the same sign, then
$|a^m-b^m|\ge|a-b|^m$.  If they have opposite signs, convexity gives
$|a-b|^m\le2^{m-1}(|a|^m+|b|^m)
=2^{m-1}|a^m-b^m|$.  Thus
$|a-b|^m\le2^{m-1}|a^m-b^m|$ in both cases.  Squaring, integrating,
and taking square roots yields \Cref{eq:power-lower}.

Suppose now that $m$ is even, fix $a,b\in\R$, and put $d=|a-b|$.
If $|a|\ge d/4$, then
$d^{2m}\le4^m|a|^md^m$, which gives
\Cref{eq:scalar-lifting}.  If $|a|<d/4$, then $|b|>3d/4$ and
\[
 |a^m-b^m|
 \ge
 \left[\left(\frac34\right)^m-\left(\frac14\right)^m\right]d^m
 \ge 2^{-m}d^m,
\]
where the last inequality follows from $3^m-1\ge2^m$.  Hence
$d^{2m}\le4^m|a^m-b^m|^2$, proving
\Cref{eq:scalar-lifting}.
\end{proof}

Fix an integer $m\ge1$ and a closed subspace
$X\subseteq L_{2m}(\mu)$.  Define
\begin{equation}\label{eq:feature-space}
 \cH_m(X)
 :=
 \overline{\spanop}\{x^m:x\in X\}
 \subseteq L_2(\mu).
\end{equation}
Equipped with the real inner product
$\langle f,h\rangle_{L_2}:=\int_\Omega fh\,d\mu$
for $f,h\in\cH_m(X)$, this is a Hilbert space.

For $x,y\in X$, the coefficient of $t^m$ in the polynomial
$t\mapsto\|x+ty\|_{2m}^{2m}$ is
\[
 \binom{2m}{m}\langle x^m,y^m\rangle_{L_2}.
\]
Consequently, every $U\in O(X)$ preserves these inner products.
This yields the following lifting lemma.

\begin{lemma}\label{lem:lifting}
Let $m\ge1$ be an integer, let $X$ be a closed subspace of
$L_{2m}(\mu)$, and let $\cH_m(X)$ be defined by
\Cref{eq:feature-space}.  For every $U\in O(X)$, there exists a unique
operator
\[
 \Pi_m(U)\in O(\cH_m(X))
\]
such that
\begin{equation}\label{eq:lift-action}
 \Pi_m(U)(x^m)=(Ux)^m
 \qquad(x\in X).
\end{equation}
Moreover, the map
\[
 \Pi_m:O(X)\longrightarrow O(\cH_m(X)),
 \qquad U\longmapsto\Pi_m(U),
\]
is a group homomorphism.  Consequently, if
$\rho:G\to O(X)$ is a strongly continuous linear isometric
representation, then
\[
 \pi=\Pi_m\circ\rho:G\longrightarrow O(\cH_m(X))
\]
is a strongly continuous orthogonal representation.
\end{lemma}

\begin{proof}
Fix $U\in O(X)$.  Since
$\|Ux+tUy\|_{2m}^{2m}=\|x+ty\|_{2m}^{2m}$ for $x,y\in X$ and
$t\in\R$, the corresponding polynomials in $t$ coincide.  Comparing
their coefficients of $t^m$ gives
\begin{equation}\label{eq:gram-invariance}
 \langle (Ux)^m,(Uy)^m\rangle_{L_2}
 =\langle x^m,y^m\rangle_{L_2}.
\end{equation}

Let $\cH_m^0(X)=\spanop\{x^m:x\in X\}$, which is dense in
$\cH_m(X)$.  For $n\ge1$, $a_1,\ldots,a_n\in\R$, and
$x_1,\ldots,x_n\in X$, define
\begin{equation}\label{eq:algebraic-lift}
 \Pi_m^0(U)\left(\sum_{j=1}^n a_jx_j^m\right)
 :=
 \sum_{j=1}^n a_j(Ux_j)^m.
\end{equation}
By \Cref{eq:gram-invariance},
\[
 \left\|
 \sum_{j=1}^n a_j(Ux_j)^m
 \right\|_2^2
 =
 \sum_{i,j=1}^n
 a_i a_j
 \langle x_i^m,x_j^m\rangle_{L_2}
 =
 \left\|
 \sum_{j=1}^n a_jx_j^m
 \right\|_2^2.
\]
In particular, every combination representing zero is sent to zero.
Thus \Cref{eq:algebraic-lift} is well-defined and defines a linear
isometry on $\cH_m^0(X)$.  It therefore extends uniquely to a linear
isometry
\[
 \Pi_m(U):\cH_m(X)\longrightarrow\cH_m(X).
\]
By construction,
\[
 \Pi_m(U)(x^m)=\Pi_m^0(U)(x^m)=(Ux)^m,
\]
so \Cref{eq:lift-action} holds.

Applying the same construction to $U^{-1}$ gives
\[
 \Pi_m(U^{-1})\Pi_m(U)(x^m)=x^m
 =\Pi_m(U)\Pi_m(U^{-1})(x^m)
 \qquad(x\in X).
\]
Since the pure powers span a dense subspace of $\cH_m(X)$, the two
operators are mutual inverses on $\cH_m(X)$.  Thus
$\Pi_m(U)\in O(\cH_m(X))$.  Any operator satisfying
\Cref{eq:lift-action} agrees with $\Pi_m(U)$ on the dense subspace
$\cH_m^0(X)$, which proves uniqueness.

For $U,V\in O(X)$ and $x\in X$,
\[
 \Pi_m(U)\Pi_m(V)(x^m)=(UVx)^m=\Pi_m(UV)(x^m).
\]
Another density argument gives
$\Pi_m(U)\Pi_m(V)=\Pi_m(UV)$.  Hence $\Pi_m$ is a group
homomorphism.

Finally, let $\rho:G\to O(X)$ be strongly continuous and put
$\pi=\Pi_m\circ\rho$.  This is an orthogonal representation.  For
$x\in X$ and $g,g_0\in G$, \Cref{eq:power-upper-general} gives
\[
 \|\pi(g)x^m-\pi(g_0)x^m\|_2
 =\|(\rho(g)x)^m-(\rho(g_0)x)^m\|_2
 \le
 m\|x\|_{2m}^{m-1}
 \|\rho(g)x-\rho(g_0)x\|_{2m}.
\]
The right-hand side tends to zero as $g\to g_0$.  Hence the orbit
maps are continuous on the dense subspace $\cH_m^0(X)$.  If
$\xi\in\cH_m(X)$ and $\eta\in\cH_m^0(X)$, then
\[
 \|\pi(g)\xi-\pi(g_0)\xi\|_2
 \le
 2\|\xi-\eta\|_2
 +\|\pi(g)\eta-\pi(g_0)\eta\|_2.
\]
Approximating $\xi$ by such vectors $\eta$ proves continuity of its
orbit map at $g_0$.  Thus $\pi$ is strongly continuous.
\end{proof}

\subsection{Spectral-gap constants and reduction of the exponent}

We now use the Hilbert-space lift constructed above to obtain uniform
spectral-gap estimates.  For the remainder of this section, fix a
locally compact second countable group $G$ with property~$(T)$ and a
Hilbert Kazhdan pair $(Q,\kappa_2)$ for $G$, in the sense of
\Cref{eq:hilbert-gap}.

For $1<s<\infty$, define the $Q$-spectral-gap constant of $G$ on
closed subspaces of $L_s$-spaces by
\begin{equation}\label{eq:Kr-definition}
 \mathsf K_s=\mathsf K_s(G,Q):=
 \inf
 \frac{\displaystyle
       \max_{q\in Q}\|\rho(q)x-x\|_s}
      {\displaystyle
       \distop\bigl(x,X^{\rho(G)}\bigr)}.
\end{equation}
Here the infimum is taken jointly over all tuples
\[
 ((\Omega,\Sigma,\mu),X,\rho,x)
\]
satisfying the following conditions:
\begin{itemize}
\item $(\Omega,\Sigma,\mu)$ is a $\sigma$-finite measure space;
\item $X$ is a closed subspace of $L_s(\mu)$;
\item $\rho:G\to O(X)$ is a strongly continuous linear isometric
      representation;
\item $x\in X\setminus X^{\rho(G)}$, where
\[
 X^{\rho(G)}
 =
 \{y\in X:\rho(g)y=y\text{ for every }g\in G\}.
\]
\end{itemize}

The quotient in \Cref{eq:Kr-definition} is well-defined.  Indeed, the
maximum in the numerator exists because $Q$ is compact and the orbit
map $q\mapsto\rho(q)x$ is norm-continuous.  Moreover,
$X^{\rho(G)}$ is closed, so the denominator is positive.  We retain
the convention $\inf\varnothing=+\infty$ introduced above and use
$1/(+\infty)=0$ whenever this case occurs below.
When admissible tuples exist, $\mathsf K_s(G,Q)$ is the optimal
constant in the corresponding uniform spectral-gap inequality.  Its
positivity is not part of the definition: a priori,
$\mathsf K_s(G,Q)$ may be zero.

At exponent $2$, every closed subspace of an $L_2$-space is a Hilbert
space, and every linear isometric representation on it is
orthogonal.  Hence \Cref{eq:hilbert-gap}, applied to each admissible
tuple in \Cref{eq:Kr-definition}, gives
\begin{equation}\label{eq:K2-lower}
 \mathsf K_2\ge\kappa_2.
\end{equation}

We shall prove that $\mathsf K_s>0$ for every even integer $s\ge2$.
The first step treats $s=2m$ with $m$ odd; these estimates will serve
as the initial cases for the exponent-reduction argument.

\begin{lemma}\label{lem:odd}
For every odd integer $m\ge1$,
\begin{equation}\label{eq:odd-constant}
 \mathsf K_{2m}
 \ge
 \frac{\kappa_2}{m2^m}.
\end{equation}
\end{lemma}

\begin{proof}
If the class of tuples in \Cref{eq:Kr-definition} at exponent $2m$ is
empty, the conclusion follows from the convention made after that
definition.  Otherwise, fix a tuple
\[
 ((\Omega,\Sigma,\mu),X,\rho,x)
\]
occurring in the infimum in \Cref{eq:Kr-definition} at exponent $2m$.
Put
\[
 F=X^{\rho(G)}
 \qquad\text{and}\qquad
 d=\distop(x,F).
\]
Since $L_{2m}(\mu)$ is uniformly convex, so is its closed subspace
$X$.  By \Cref{lem:normalization}, there exists $a\in F$ such that
$\|x-a\|_{2m}=d$ and the vector
\[
 v=\frac{x-a}{d}
\]
satisfies
\[
 \|v\|_{2m}
 =
 \distop(v,F)
 =
 1,
 \qquad
 \sup_{g\in G}\|\rho(g)v-v\|_{2m}\ge1.
\]
Since $a\in F$, for every $g\in G$ we have
\[
 \|\rho(g)v-v\|_{2m}
 =
 \frac{\|\rho(g)x-x\|_{2m}}{d}.
\]

Let
\[
 \pi=\Pi_m\circ\rho:G\longrightarrow O(\cH_m(X))
\]
be the lifted representation from \Cref{lem:lifting}, and put
$u=v^m\in\cH_m(X)$.  Then $\|u\|_2=1$.

By \Cref{eq:lift-action,eq:power-lower}, for every $g\in G$,
\[
\begin{aligned}
 \|\pi(g)u-u\|_2
 &=
 \|(\rho(g)v)^m-v^m\|_2\\
 &\ge
 2^{-(m-1)}
 \|\rho(g)v-v\|_{2m}^m.
\end{aligned}
\]
Taking the supremum over $g\in G$ and using
the normalization of $v$, we obtain
\[
 \sup_{g\in G}\|\pi(g)u-u\|_2
 \ge
 2^{-(m-1)}.
\]

For $h\in\cH_m(X)^{\pi(G)}$ and $g\in G$, orthogonality of
$\pi(g)$ gives
\[
 \|\pi(g)u-u\|_2
 \le
 \|\pi(g)(u-h)\|_2+\|u-h\|_2
 =2\|u-h\|_2.
\]
Taking the supremum over $g$ and then the infimum over $h$, and using
the preceding lower bound, yields
\begin{equation}\label{eq:odd-Hilbert-distance}
 \distop\bigl(u,\cH_m(X)^{\pi(G)}\bigr)
 \ge
 \frac12\sup_{g\in G}\|\pi(g)u-u\|_2
 \ge 2^{-m}.
\end{equation}

On the other hand, $\|v\|_{2m}=\|\rho(q)v\|_{2m}=1$ for every
$q\in Q$.  Applying \Cref{eq:hilbert-gap} to $\pi$ and $u$, and
then using the preceding estimate and
\Cref{eq:power-upper-general}, we obtain
\begin{equation}\label{eq:odd-displacement-lower}
\begin{aligned}
 \kappa_2 2^{-m}
 &\le
 \kappa_2\,
 \distop\bigl(u,\cH_m(X)^{\pi(G)}\bigr)\\
 &\le
 \max_{q\in Q}\|\pi(q)u-u\|_2\\
 &=
 \max_{q\in Q}\|(\rho(q)v)^m-v^m\|_2\\
 &\le
 m\max_{q\in Q}\|\rho(q)v-v\|_{2m}.
\end{aligned}
\end{equation}
Finally, the displacement identity above gives
\[
 \frac{\max_{q\in Q}\|\rho(q)x-x\|_{2m}}
      {\distop\bigl(x,X^{\rho(G)}\bigr)}
 =
 \max_{q\in Q}\|\rho(q)v-v\|_{2m}
 \ge
 \frac{\kappa_2}{m2^m}.
\]
Taking the infimum over all tuples in
\Cref{eq:Kr-definition} proves the lemma.
\end{proof}

We now turn to the remaining case in which $m\ge2$ is even.  The
argument uses the fact that every $m$th power is nonnegative.  The
next lemma shows that its orthogonal projection onto the fixed vectors
of the lifted representation remains nonnegative, and estimates the
projection error.  This invariant function will define a weighted
measure through which the exponent is reduced from $2m$ to $m$.

\begin{lemma}
\label{lem:positive-projection}
Let $m\ge2$ be an even integer.  Let
$(\Omega,\Sigma,\mu)$ be a $\sigma$-finite measure space, let $X$ be
a closed subspace of $L_{2m}(\mu)$, and let
\[
 \rho:G\longrightarrow O(X)
\]
be a strongly continuous linear isometric representation.  Let
\[
 \pi=\Pi_m\circ\rho:
 G\longrightarrow O(\cH_m(X))
\]
be the lifted representation from \Cref{lem:lifting}, and let $P$ be
the orthogonal projection of $\cH_m(X)$ onto
$\cH_m(X)^{\pi(G)}$.  Then, for every $v\in S(X)$,
\[
 P(v^m)\ge0
 \qquad\text{$\mu$-almost everywhere},
\]
and
\begin{equation}\label{eq:positive-projection-error}
 \|v^m-P(v^m)\|_2
 \le
 \frac{m}{\kappa_2}
 \max_{q\in Q}\|\rho(q)v-v\|_{2m}.
\end{equation}
\end{lemma}
\begin{proof}
Fix $v\in S(X)$, and set $u=v^m$ and
$F=\cH_m(X)^{\pi(G)}$.  By \Cref{lem:lifting}, $\pi$ is a strongly
continuous orthogonal representation on $\cH_m(X)$.  Since $m$ is
even, for every $g\in G$,
\[
 \pi(g)u=(\rho(g)v)^m=|\rho(g)v|^m\ge0
 \qquad\text{$\mu$-almost everywhere}.
\]
Thus the orbit of $u$ consists of nonnegative functions.

Let
\[
 C=\overline{\operatorname{co}}^{\,\|\cdot\|_2}
   \{\pi(g)u:g\in G\}
 \subseteq\cH_m(X).
\]
The subspace $F^\perp$ is $\pi(G)$-invariant, because $\pi$ is
orthogonal and fixes $F$ pointwise.  Consequently,
\[
 P\pi(g)=P
 \qquad(g\in G).
\]
It follows by linearity and continuity of $P$ that
\[
 Pc=Pu
 \qquad(c\in C),
\]
and hence $C\subseteq Pu+F^\perp$.

By the Hilbert-space projection theorem, $C$ has a unique element
$y$ of minimal norm.  Since $\pi(g)C=C$ and $\pi(g)$ is orthogonal,
uniqueness gives $\pi(g)y=y$ for every $g\in G$; thus $y\in F$.
On the other hand, $y\in C\subseteq Pu+F^\perp$.  Since $Pu\in F$,
we conclude that
\begin{equation}\label{eq:positive-projection-in-hull}
 y=Pu=P(v^m)\in C.
\end{equation}

The positive cone
\[
 L_2^+(\mu)
 :=\{f\in L_2(\mu):f\ge0\ \text{$\mu$-almost everywhere}\}
\]
is closed and convex.  Since the orbit of $u$ is contained in
$L_2^+(\mu)$, so is $C$.  Hence
\Cref{eq:positive-projection-in-hull} gives
\[
 P(v^m)\ge0
 \qquad\text{$\mu$-almost everywhere}.
\]

Finally, since $P$ is the orthogonal projection onto $F$,
\[
 \distop(v^m,F)=\|v^m-P(v^m)\|_2.
\]
Applying
\Cref{eq:hilbert-gap,eq:lift-action,eq:power-upper-general} and using
$\|\rho(q)v\|_{2m}=\|v\|_{2m}=1$, we obtain
\begin{equation}\label{eq:positive-projection-gap}
\begin{aligned}
 \kappa_2\|v^m-P(v^m)\|_2
 &\le
 \max_{q\in Q}\|\pi(q)(v^m)-v^m\|_2\\
 &=
 \max_{q\in Q}\|(\rho(q)v)^m-v^m\|_2\\
 &\le
 m\max_{q\in Q}\|\rho(q)v-v\|_{2m}.
\end{aligned}
\end{equation}
This proves \Cref{eq:positive-projection-error}.
\end{proof}

The next lemma is the key exponent-reduction step.  The invariant
function obtained in \Cref{lem:positive-projection} is used as a
density, and the original action descends to a closed subspace of the
resulting weighted $L_m$-space.

\begin{lemma}\label{lem:descent}
Under the assumptions and notation of
\Cref{lem:positive-projection}, fix $v\in S(X)$, put $z=P(v^m)$,
and define the measure $\nu$ on $(\Omega,\Sigma)$ by $d\nu=z\,d\mu$.
Let
\[
 J_z:X\longrightarrow L_m(\nu),
 \qquad
 J_zx=[x]_\nu,
 \qquad
 E_z=\overline{J_z(X)}^{\,L_m(\nu)}.
\]
Then $\nu$ is $\sigma$-finite, $J_z$ is a linear contraction, and
there exists a unique strongly continuous linear isometric representation
\[
 \rho_z:G\longrightarrow O(E_z)
\]
such that
\begin{equation}\label{eq:descended-action}
 \rho_z(g)J_zx=J_z(\rho(g)x)
 \qquad(g\in G,\ x\in X).
\end{equation}
\end{lemma}
\begin{proof}
By \Cref{lem:positive-projection}, $z\ge0$ almost everywhere.  Since
$P$ is an orthogonal projection and $v\in S(X)$,
$\|z\|_2\le\|v^m\|_2=1$.  Choose measurable sets $A_n\subseteq\Omega$
such that $\mu(A_n)<\infty$ and $\Omega=\bigcup_{n=1}^\infty A_n$.
Then Cauchy--Schwarz gives
\[
 \nu(A_n)=\int_{A_n}z\,d\mu
 \le \|z\|_2\mu(A_n)^{1/2}<\infty,
\]
so $\nu$ is $\sigma$-finite.

Since $\nu\ll\mu$, every $\mu$-equivalence class $x\in X$ determines
a unique $\nu$-equivalence class, denoted by $[x]_\nu$.  Moreover,
H\"older's inequality gives
\begin{equation}\label{eq:weighted-continuity}
 \|J_zx\|_{L_m(\nu)}^m
 =\int_\Omega z|x|^m\,d\mu
 \le \|z\|_2\bigl\||x|^m\bigr\|_2
 \le \|x\|_{2m}^m.
\end{equation}
Thus $J_zx\in L_m(\nu)$ for every $x\in X$, and $J_z$ is a
well-defined linear contraction.

Since $m$ is even, $|x|^m=x^m$.  Using \Cref{eq:lift-action}, the
orthogonality of $\pi(g)$, and the $\pi(G)$-invariance of $z$, we
obtain, for every $g\in G$ and $x\in X$,
\begin{equation}\label{eq:weighted-invariance}
\begin{aligned}
 \|J_z(\rho(g)x)\|_{L_m(\nu)}^m
 &=\langle z,(\rho(g)x)^m\rangle_{L_2}
  =\langle z,\pi(g)x^m\rangle_{L_2}\\
 &=\langle\pi(g^{-1})z,x^m\rangle_{L_2}
  =\langle z,x^m\rangle_{L_2}
  =\|J_zx\|_{L_m(\nu)}^m
\end{aligned}
\end{equation}

It follows from \Cref{eq:weighted-invariance} that the formula
\[
 T_g(J_zx)=J_z(\rho(g)x)
\]
defines a well-defined linear isometry $T_g$ on $J_z(X)$.  Its unique
continuous extension to $E_z$ will be denoted by $\rho_z(g)$.  The
extension associated with $g^{-1}$ is its inverse, and hence
$\rho_z(g)\in O(E_z)$.

For $g,h\in G$, the operators $\rho_z(gh)$ and
$\rho_z(g)\rho_z(h)$ agree on the dense subspace $J_z(X)$ and hence
on $E_z$.  Thus $\rho_z:G\to O(E_z)$ is a linear isometric
representation satisfying \Cref{eq:descended-action}.

It remains to verify strong continuity.  For $x\in X$ and
$g,g_0\in G$, \Cref{eq:weighted-continuity} gives
\[
 \|\rho_z(g)J_zx-\rho_z(g_0)J_zx\|_{L_m(\nu)}
 \le
 \|\rho(g)x-\rho(g_0)x\|_{2m}.
\]
The right-hand side tends to zero as $g\to g_0$ by the strong
continuity of $\rho$.  Hence the orbit maps are continuous on
$J_z(X)$.  Density of $J_z(X)$ in $E_z$, together with
$\|\rho_z(g)\|=1$ for every $g\in G$, extends this conclusion to all
vectors in $E_z$.  Therefore $\rho_z$ is strongly continuous.

Finally, any representation satisfying
\Cref{eq:descended-action} agrees with $\rho_z$ on the dense subspace
$J_z(X)$ and hence on all of $E_z$.  This proves uniqueness.
\end{proof}

We now combine the projection estimate and the descended action with
\Cref{eq:scalar-lifting} to obtain the recursive spectral-gap bound.

\begin{lemma}\label{lem:recurrence}
Let $m\ge2$ be an even integer.  If $\mathsf K_m>0$, then
\begin{equation}\label{eq:recurrence}
 \mathsf K_{2m}\ge
 \min\left\{
  \frac{\kappa_2}{2^{m+1}m\sqrt{3}},
  \frac{\mathsf K_m}{8\,3^{1/m}},
  \frac{\kappa_2}{3m8^m}
 \right\}.
\end{equation}
In particular, $\mathsf K_{2m}>0$.
\end{lemma}

\begin{proof}
If there is no admissible tuple at exponent $2m$, then
$\mathsf K_{2m}=+\infty$ and the assertion is immediate.  Otherwise,
fix an admissible tuple
$((\Omega,\Sigma,\mu),X,\rho,x)$ in the definition of
$\mathsf K_{2m}$.  Apply \Cref{lem:normalization} to $x$, and denote
the resulting normalized vector by $v$.  Set
\[
 \delta=\max_{q\in Q}\|\rho(q)v-v\|_{2m},
 \qquad
 D=\sup_{g\in G}\|\rho(g)v-v\|_{2m}.
\]
The normalization of $v$ and the equality $\|v\|_{2m}=1$ give
$1\le D\le2$.

Let $\pi=\Pi_m\circ\rho$ be the lifted representation.  Let $z$ be
the orthogonal projection of $v^m$ onto
$\cH_m(X)^{\pi(G)}$.  By \Cref{lem:positive-projection}, the function
$z$ is nonnegative and fixed by $\pi(G)$.  Moreover,
\Cref{eq:positive-projection-error} gives
\[
 \|v^m-z\|_2
 \le\frac{m}{\kappa_2}\delta.
\]
It follows that
\begin{equation}\label{eq:power-global}
 \|(\rho(g)v)^m-v^m\|_2
 =
 \|\pi(g)(v^m-z)-(v^m-z)\|_2
 \le\frac{2m}{\kappa_2}\delta
 \qquad(g\in G).
\end{equation}

Apply \Cref{lem:descent} with this $v$.  By
\Cref{eq:descended-action} and the contractivity of $J_z$,
\[
 \max_{q\in Q}
 \|\rho_z(q)J_zv-J_zv\|_{E_z}
 \le\delta.
\]
If $J_zv\notin E_z^{\rho_z(G)}$, then
\[
 ((\Omega,\Sigma,\nu),E_z,\rho_z,J_zv)
\]
is an admissible tuple at exponent $m$, and the definition of
$\mathsf K_m$ gives
\[
 \distop\bigl(J_zv,E_z^{\rho_z(G)}\bigr)
 \le\frac{\delta}{\mathsf K_m},
\]
whereas the same estimate is immediate if $J_zv$ is fixed.  Notice
that the former case cannot occur when $\mathsf K_m=+\infty$, since it
would give an admissible tuple with a finite quotient in
\Cref{eq:Kr-definition}.  Hence, for every $g\in G$,
\begin{equation}\label{eq:weighted-global}
 \|\rho_z(g)J_zv-J_zv\|_{E_z}
 \le
 2\,\distop\bigl(J_zv,E_z^{\rho_z(G)}\bigr)
 \le\frac{2\delta}{\mathsf K_m}.
\end{equation}

Decomposing $v^m=z+(v^m-z)$ and using H\"older's inequality,
\Cref{eq:positive-projection-error}, and
\Cref{eq:weighted-global}, we obtain
\begin{align}
 \int_\Omega v^m|\rho(g)v-v|^m\,d\mu
 &=
 \int_\Omega z|\rho(g)v-v|^m\,d\mu
 +\int_\Omega(v^m-z)|\rho(g)v-v|^m\,d\mu \notag\\
 &\le
 \|\rho_z(g)J_zv-J_zv\|_{E_z}^m
 +\|v^m-z\|_2\|\rho(g)v-v\|_{2m}^m \notag\\
 &\le
 \left(\frac{2\delta}{\mathsf K_m}\right)^m
 +\frac{m\delta}{\kappa_2}
  \|\rho(g)v-v\|_{2m}^m.                 \label{eq:weighted-control}
\end{align}

Apply \Cref{eq:scalar-lifting} pointwise with
$a=v(\omega)$ and $b=(\rho(g)v)(\omega)$, and then integrate.  Since
$m$ is even, $|v|^m=v^m$.  Together with
\Cref{eq:power-global,eq:weighted-control}, this gives
\begin{equation}\label{eq:master}
\begin{aligned}
 \|\rho(g)v-v\|_{2m}^{2m}
 &\le4^m\left[
  \|(\rho(g)v)^m-v^m\|_2^2
  +\int_\Omega v^m|\rho(g)v-v|^m\,d\mu
 \right]\\
 &\le4^m\left[
  \left(\frac{2m\delta}{\kappa_2}\right)^2
  +\left(\frac{2\delta}{\mathsf K_m}\right)^m
  +\frac{m\delta}{\kappa_2}
   \|\rho(g)v-v\|_{2m}^m
 \right]
 \qquad(g\in G).
\end{aligned}
\end{equation}

Taking the supremum over $g\in G$ in \Cref{eq:master} and using
$1\le D\le2$, we obtain
\begin{equation}\label{eq:master-sup}
 1\le D^{2m}\le4^m\left[
  \left(\frac{2m\delta}{\kappa_2}\right)^2
  +\left(\frac{2\delta}{\mathsf K_m}\right)^m
  +\frac{m2^m\delta}{\kappa_2}
 \right].
\end{equation}

Suppose, toward a contradiction, that $\delta$ is strictly smaller
than the minimum in \Cref{eq:recurrence}.  Then
\[
 \left(\frac{2m\delta}{\kappa_2}\right)^2
 <\frac{1}{3\cdot4^m},
 \qquad
 \left(\frac{2\delta}{\mathsf K_m}\right)^m
 <\frac{1}{3\cdot4^m},
\]
and
\[
 \frac{m2^m\delta}{\kappa_2}
 <\frac{1}{3\cdot4^m}.
\]
Thus the right-hand side of \Cref{eq:master-sup} would be strictly
smaller than $1$, a contradiction.  Hence $\delta$ is at least the
minimum in \Cref{eq:recurrence}.

Finally, since $v=(x-a)/d$, where $a\in X^{\rho(G)}$ and
$d=\distop(x,X^{\rho(G)})$, we have
\[
 \delta
 =
 \frac{\max_{q\in Q}\|\rho(q)x-x\|_{2m}}
 {\distop\bigl(x,X^{\rho(G)}\bigr)}.
\]
Since the original admissible data were arbitrary,
\Cref{eq:recurrence} follows.
\end{proof}

\section{Proofs of the main results}\label{sec:main-proofs}

We now prove the two results stated in the introduction.  The
quantitative theorem follows by iterating the exponent-reduction
estimate, while the corollary combines this theorem with the
nonexceptional case of Bader, Furman, Gelander and Monod and their
duality principle.

\subsection{The uniform spectral-gap estimate}

\begin{proof}[Proof of \Cref{thm:uniform-subspace}]
Fix $G$, the Hilbert Kazhdan pair $(Q,\kappa_2)$, and an even
integer $p\ge2$ as in \Cref{thm:uniform-subspace}.  Throughout this
proof, write
\[
 \mathsf K_s=\mathsf K_s(G,Q).
\]
Factor $p$ uniquely as
\[
 p=2^a r,
 \qquad
 a\ge1,
 \qquad
 r\ \text{odd},
\]
and, for $1\le j\le a$, set
\[
 p_j=2^jr.
\]
Thus
\[
 p_1=2r,
 \qquad
 p_a=p.
\]
We construct the constant $c_p$ along the chain
\[
 p_1,p_2,\ldots,p_a.
\]

First define
\[
 c_{p_1}=c_{2r}=
 \begin{cases}
  \kappa_2, & r=1,\\[2mm]
  \displaystyle\frac{\kappa_2}{r2^r}, & r>1.
 \end{cases}
\]
If $r=1$, then $p_1=2$, and \Cref{eq:K2-lower} gives
\[
 \mathsf K_{p_1}
 =\mathsf K_2
 \ge\kappa_2
 =c_{p_1}.
\]
If $r>1$, then \Cref{lem:odd}, applied with $m=r$, gives
\[
 \mathsf K_{p_1}
 =\mathsf K_{2r}
 \ge\frac{\kappa_2}{r2^r}
 =c_{p_1}.
\]
Thus, in either case,
\begin{equation}\label{eq:initial-gap}
 \mathsf K_{p_1}\ge c_{p_1}>0.
\end{equation}

For $1\le j<a$, define recursively
\[
 c_{p_{j+1}}
 =
 \min\left\{
  \frac{\kappa_2}{2^{p_j+1}p_j\sqrt{3}},
  \frac{c_{p_j}}{8\,3^{1/p_j}},
  \frac{\kappa_2}{3p_j8^{p_j}}
 \right\}.
\]
We claim that
\begin{equation}\label{eq:inductive-gap}
 \mathsf K_{p_j}\ge c_{p_j}>0
 \qquad(1\le j\le a).
\end{equation}
The case $j=1$ is exactly \Cref{eq:initial-gap}.  Suppose that
\Cref{eq:inductive-gap} holds for some $j<a$.  Since
$p_j=2^jr$ and $j\ge1$, the integer $p_j$ is even and satisfies
$p_j\ge2$.  Moreover, the inductive hypothesis gives
$\mathsf K_{p_j}>0$.  Hence \Cref{lem:recurrence}, applied with
$m=p_j$, yields
\begin{align*}
 \mathsf K_{p_{j+1}}
 &=\mathsf K_{2p_j}\\
 &\ge
 \min\left\{
  \frac{\kappa_2}{2^{p_j+1}p_j\sqrt{3}},
  \frac{\mathsf K_{p_j}}{8\,3^{1/p_j}},
  \frac{\kappa_2}{3p_j8^{p_j}}
 \right\}\\
 &\ge
 \min\left\{
  \frac{\kappa_2}{2^{p_j+1}p_j\sqrt{3}},
  \frac{c_{p_j}}{8\,3^{1/p_j}},
  \frac{\kappa_2}{3p_j8^{p_j}}
 \right\}\\
 &=c_{p_{j+1}}>0.
\end{align*}
This proves \Cref{eq:inductive-gap} by induction.  Since $p_a=p$,
we obtain
\begin{equation}\label{eq:Kp-positive}
 \mathsf K_p\ge c_p>0.
\end{equation}
Because $a$, $r$, and the sequence $(p_j)_{j=1}^a$ are determined
uniquely by $p$, the constant $c_p$ depends only on $p$ and
$\kappa_2$.

Now let $(\Omega,\Sigma,\mu)$, $X$, $\rho$, and $x$ be as in
\Cref{thm:uniform-subspace}.  If $x\notin X^{\rho(G)}$, then these
data form an admissible tuple in the definition of $\mathsf K_p$.
Consequently, \Cref{eq:Kr-definition,eq:Kp-positive} give
\begin{align*}
 \max_{q\in Q}\|\rho(q)x-x\|_p
 &\ge
 \mathsf K_p\,
 \distop\bigl(x,X^{\rho(G)}\bigr)\\
 &\ge
 c_p\,
 \distop\bigl(x,X^{\rho(G)}\bigr).
\end{align*}
If $x\in X^{\rho(G)}$, then both sides of
\Cref{eq:main-gap} vanish.  Therefore \Cref{eq:main-gap} holds for
every $x\in X$, and the proof is complete.
\end{proof}

\subsection{Subspaces and quotients}

\begin{proof}[Proof of \Cref{cor:main}]
We first prove the closed-subspace assertion.  Let
$X\subseteq L_p(\mu)$ be a closed subspace.  If $p$ is not an even
integer, the conclusion follows from the result of Bader, Furman,
Gelander and Monod \cite[Theorem~A(ii)]{BFGM}.

Suppose that $p$ is even.  Choose a Hilbert Kazhdan pair
$(Q,\kappa_2)$ for $G$, and let
\[
 \rho:G\longrightarrow O(X)
\]
be an arbitrary strongly continuous linear isometric representation.
As a closed subspace of $L_p(\mu)$, the space $X$ is uniformly convex
and uniformly smooth.  Hence \Cref{lem:canonical-complement} gives a
closed $\rho(G)$-invariant subspace $X'_\rho$ such that
\[
 X=X^{\rho(G)}\oplus X'_\rho,
 \qquad
 (X'_\rho)^{\rho(G)}=\{0\}.
\]
Moreover, the quotient representation on $X/X^{\rho(G)}$ almost has
invariant vectors if and only if the restriction of $\rho$ to
$X'_\rho$ does.

If $X'_\rho=\{0\}$, then $X/X^{\rho(G)}=\{0\}$, so the quotient
representation does not almost have invariant vectors.  Suppose that
$X'_\rho\ne\{0\}$, and put
\[
 \rho'=\rho|_{X'_\rho}.
\]
Then $\rho':G\to O(X'_\rho)$ is a strongly continuous linear
isometric representation and
\[
 (X'_\rho)^{\rho'(G)}=\{0\}.
\]
Therefore \Cref{thm:uniform-subspace}, applied to $X'_\rho$ and
$\rho'$, gives
\[
 \max_{q\in Q}\|\rho'(q)v-v\|_p
 \ge c_p\|v\|_p
 \qquad(v\in X'_\rho).
\]
Let $e$ denote the identity element of $G$.  Since $e\in Q$, every
$v\in S(X'_\rho)$ satisfies
\begin{align*}
 \operatorname{diam}\bigl(\rho'(Q)v\bigr)
 &\ge
 \max_{q\in Q}
 \|\rho'(q)v-\rho'(e)v\|_p\\
 &=
 \max_{q\in Q}\|\rho'(q)v-v\|_p\\
 &\ge c_p.
\end{align*}
Consequently,
\[
 \inf_{v\in S(X'_\rho)}
 \operatorname{diam}\bigl(\rho'(Q)v\bigr)
 \ge c_p>0.
\]
By \Cref{eq:almost-invariant}, the representation $\rho'$ does not
almost have invariant vectors.  It follows from
\Cref{lem:canonical-complement} that neither does the quotient
representation on $X/X^{\rho(G)}$.  Since $\rho$ was arbitrary, $G$
has property $(T_X)$.  This proves the closed-subspace assertion for
every $1<p<\infty$.

We now prove the quotient assertion.  Let $N\subseteq L_p(\mu)$ be a
closed subspace and put
\[
 B=L_p(\mu)/N.
\]
If $B=\{0\}$, then property $(T_B)$ is immediate.  Assume that
$B\ne\{0\}$, and set
\[
 p'=\frac{p}{p-1}.
\]
Let
\[
 q:L_p(\mu)\longrightarrow B
\]
be the quotient map.  Under the canonical isometric identification
$L_p(\mu)^*=L_{p'}(\mu)$, its adjoint
\[
 q^*:B^*\longrightarrow L_{p'}(\mu)
\]
is an isometric embedding whose range is
\begin{equation}\label{eq:quotient-dual}
 N^\perp
 =
 \left\{
  f\in L_{p'}(\mu):
  \int_\Omega yf\,d\mu=0
  \text{ for every }y\in N
 \right\}.
\end{equation}
Thus
\[
 q^*:B^*\longrightarrow N^\perp
\]
is a surjective linear isometry, and $N^\perp$ is a closed subspace
of $L_{p'}(\mu)$.  The closed-subspace assertion already proved,
applied at exponent $p'$, gives property $(T_{N^\perp})$.

Conjugation by the isometry $q^*$ identifies the strongly continuous
linear isometric representations on $B^*$ with those on $N^\perp$.
Under this identification, fixed-vector spaces, quotient
representations, and the almost-invariant-vector condition are
preserved.  Hence $G$ has property $(T_{B^*})$.

Finally, uniform convexity and uniform smoothness pass to quotients,
so $B=L_p(\mu)/N$ is uniformly convex and uniformly smooth.  The
duality theorem of Bader, Furman, Gelander and Monod
\cite[Corollary~2.12]{BFGM} therefore gives
\[
 (T_B)\quad\Longleftrightarrow\quad(T_{B^*}).
\]
Since $G$ has property $(T_{B^*})$, it follows that $G$ has property
$(T_B)$.  This proves the quotient assertion and completes the proof.
\end{proof}


\begin{thebibliography}{9}

\bibitem{BFGM}
U.~Bader, A.~Furman, T.~Gelander, and N.~Monod,
\emph{Property $(T)$ and rigidity for actions on Banach spaces},
Acta Math. \textbf{198} (2007), no.~1, 57--105.

\bibitem{BaderICM}
U.~Bader,
\emph{Higher property $T$ and below-rank phenomena of lattices---an
extract},
in \emph{Proceedings of the International Congress of Mathematicians
2026, Volume~4: Invited Lectures (Sections~5--8)},
S.~Friedlander and Y.~Tschinkel (eds.), SIAM, Philadelphia, 2026,
24--40,
doi:\href{https://doi.org/10.1137/25M1806983}{10.1137/25M1806983}.

\bibitem{BHV}
B.~Bekka, P.~de la Harpe, and A.~Valette,
\emph{Kazhdan's Property $(T)$},
New Mathematical Monographs, vol.~11, Cambridge University Press,
Cambridge, 2008.

\bibitem{BoucherSpakula}
K.~Boucher and J.~\v{S}pakula,
\emph{Sobolev spaces and uniform boundary representations},
preprint, 2023, arXiv:2306.09999.

\bibitem{deLaSalleICM}
M.~de la Salle,
\emph{Analysis on simple Lie groups and lattices},
in \emph{ICM---International Congress of Mathematicians, Vol.~4:
Sections~5--8},
EMS Press, Berlin, 2023, 3166--3188,
doi:\href{https://doi.org/10.4171/ICM2022/160}{10.4171/ICM2022/160}.

\bibitem{Hardin}
C.~D.~Hardin, Jr.,
\emph{Isometries on subspaces of $L_p$},
Indiana Univ. Math. J. \textbf{30} (1981), no.~3, 449--465.

\end{thebibliography}
\end{document}